\documentclass[conference]{IEEEtran}
\usepackage{fancyhdr}
\IEEEoverridecommandlockouts
\usepackage{cite}
\usepackage{amsmath,amssymb,amsfonts}

\newcommand{\range}[1]{\operatorname{range}\!\left(#1\right)}

\usepackage{amsthm}

\newtheorem{theorem}{Theorem}[section]
\newtheorem{lemma}[theorem]{Lemma}

\usepackage{algorithm}
\usepackage{algpseudocode}
\usepackage{graphicx}
\usepackage{textcomp}
\usepackage{xcolor}
\usepackage{nicematrix}

\def\BibTeX{{\rm B\kern-.05em{\sc i\kern-.025em b}\kern-.08em
    T\kern-.1667em\lower.7ex\hbox{E}\kern-.125emX}}

\usepackage{todonotes}
\setuptodonotes{inline}

\usepackage[hidelinks]{hyperref}

\usepackage{tikz}
\usetikzlibrary{arrows.meta,positioning}

\usepackage{subcaption}

\usepackage{siunitx}

\AtBeginDocument{%

}

\begin{document}

\title{A GPU-Accelerated Blocked Adaptive Randomized Range Finder
Based on an Implicit Householder QR Decomposition
}

\author{\IEEEauthorblockN{Carolin Penke}
\IEEEauthorblockA{\textit{Jülich Supercomputing Centre} \\
\textit{Forschungszentrum Jülich}\\
Jülich, Germany \\
c.penke@fz-juelich.de}
\and
\IEEEauthorblockN{Andreas Herten}
\IEEEauthorblockA{\textit{Jülich Supercomputing Centre} \\
\textit{Forschungszentrum Jülich}\\
Jülich, Germany \\
a.herten@fz-juelich.de}
}

\maketitle
\thispagestyle{fancy}
\lhead{}
\rhead{}
\chead{}
\lfoot{\footnotesize{
SC26 Workshops, November 15-20, 2026, Chicago, Illinois, USA
\newline 979-8-3195-1221-5/26/\$31.00 \copyright 2026 IEEE}}
\rfoot{}
\cfoot{}
\renewcommand{\headrulewidth}{0pt}
\renewcommand{\footrulewidth}{0pt}

\begin{abstract}
Low-rank methods can reduce the memory and computational requirements of deep neural network training in approaches such as GaLore. Randomized range finders offer an attractive alternative to singular value decompositions, particularly when the required rank is determined adaptively from a prescribed approximation tolerance. We introduce a blocked adaptive randomized range finder based on a Householder QR decomposition and an optimized hybrid CPU--GPU implementation. Numerical experiments show that it preserves orthogonality and approximation accuracy in regimes where block Gram--Schmidt without reorthogonalization becomes unstable. The blocked formulation exposes matrix--matrix operations and enables overlap of CPU panel factorization with GPU updates. On an NVIDIA GH200, the overlapped implementation reduces the runtime for the largest tested matrix from \qty{9.91}{\second} on the CPU to \qty{0.407}{\second}. The method provides a stable and efficient building block for low-rank approximation on heterogeneous systems with applications in computational science and engineering.
\end{abstract}


\section{Introduction}
% % Deep neural networks, such as GPT-like transformer architectures, are increasingly prevalent and consume significant portions of global computing infrastructure, predominantly using GPUs. These models demand vast datasets and are constrained by available compute capabilities during both the pre-training stage on supercomputers and the fine-tuning stage on smaller workstations. Enhancing training efficiency is therefore highly impactful. This work introduces techniques to leverage low-rank structures for reducing memory requirements and outlines a method to efficiently acquire  the necessary subspaces by using a randomized range finder. We propose a GPU-accelerated algorithm, based on the Householder QR decomposition that is also applicable beyond deep learning contexts.

% Where do you need dominant subspaces.

% AI Galore, as motivating example throughout the paper

% SVD

% Randomized numerical linear algebra

% GPUs

% Blocked algorithm.

Dominant low-dimensional subspaces of a matrix, and their approximations, form a crucial element in many methods for solving problems in computational science, engineering, and machine learning. 
Such a subspace, represented by an orthonormal basis in matrix form, can serve as a starting point for efficiently computing an approximate truncated singular value decomposition (SVD). The SVD forms the basis of principal component analysis, a cornerstone of data analysis and feature selection, where directions are identified that capture the largest variation in a dataset.

Our main motivating example comes from the training of deep neural networks. Modern architectures such as transformers place high demands on computation and GPU memory during pre-training and fine-tuning. Low-rank methods reduce these requirements by operating in reduced subspaces. In particular, GaLore and its successors project gradients onto a low-dimensional subspace and represent optimizer states in this reduced space, thereby reducing memory consumption~\cite{zhaoGaLoreMemoryEfficientLLM2024}. Efficiently computing a basis for such subspaces is an important component of these methods.

Further applications arise in model order reduction, e.g., proper orthogonal decomposition (POD)~\cite{POD}.

A common computational problem underlying these examples is the following. Given a matrix $A\in\mathbb{R}^{m\times n}$, the goal is to compute an orthonormal basis $Q\in\mathbb{R}^{m\times r}$ that approximately captures the dominant range of $A$, in the sense that
\begin{align*}
A &\approx QQ^TA,
\qquad
Q^TQ=I_r,
% \label{Eq:rangeQ}
\end{align*}
where $I_r$ denotes an identity matrix of size $r$. With $r<n$, this is a low-rank approximation. 

% Since $B$ has only $k$ rows, its SVD can be computed at a substantially lower cost, and the resulting singular vectors can be mapped back through $Q$ to form an approximate low-rank SVD~\cite{halkoFindingStructureRandomness2011}.

Given an SVD $A=U\Sigma V^T$, the Eckart--Young theorem gives a best rank-$r$ approximation by choosing $Q$ as the first $r$ columns of $U$, yielding $\|A-QQ^TA\|_2=\sigma_{r+1}$.
Computing these vectors through a full SVD, for example through bidiagonalization followed by a divide-and-conquer method, however, is expensive~\cite{golub13}. In many applications, only a low-dimensional dominant subspace basis $Q$ is required, so a full factorization performs more work than necessary.

Randomized numerical linear algebra provides efficient alternatives to classical matrix algorithms by incorporating random sampling while retaining numerical or probabilistic guarantees~\cite{murray2023randomizednumericallinearalgebra,halkoFindingStructureRandomness2011}. \emph{Randomized range finders} are a widely used class of methods for computing dominant subspaces. They can construct a basis either for a prescribed target rank or adaptively until a specified approximation tolerance is reached.

Blocked randomized range finders expose matrix--matrix operations that are well suited to modern memory hierarchies and accelerator architectures such as GPUs. Existing blocked adaptive formulations commonly rely on Gram--Schmidt orthogonalization and reorthogonalization~\cite{martinssonRandomizedBlockedAlgorithm2016,wenjianyuyugu+andyaohangliEfficientRandomizedAlgorithms}. Besides their numerical stability implications, the resulting dependencies can limit the ability to overlap panel processing with matrix updates in hybrid CPU--GPU implementations.

In this work, we generalize the framework of blocked Householder QR factorizations~\cite{higham2002,schreiberStorageEfficientWY1987,elmrothApplyingRecursionSerial2000} to obtain a blocked adaptive randomized range finder. By representing the basis through Householder reflectors, the resulting formulation does not require reorthogonalization. The blocked formulation exposes predominantly matrix--matrix operations. For the GPU implementation, we adapt the hybrid QR factorization strategy used in the MAGMA library~\cite{magma}, restructuring the computation to efficiently overlap CPU panel factorization with GPU-based matrix updates. Numerical experiments evaluate approximation quality and numerical stability of the proposed formulation as well as the performance of its hybrid CPU--GPU implementation on an NVIDIA GH200.

\section{Background and Related Work}\label{Sec:Background}

The simplest version of a randomized range finder~\cite{halkoFindingStructureRandomness2011} approximates the dominant range of $A$ by sampling a matrix panel $Y$ and computing an orthonormal basis $Q$ for $\range{Y}$, as outlined in \autoref{Fig:RandRangeFinder_fixed}. Statistical guarantees on the approximation quality are given in~\cite{halkoFindingStructureRandomness2011}.

\begin{figure}[H]
\centering
\fbox{%
  \begin{minipage}{0.95\linewidth}
         \begin{equation*}
          \begin{aligned}
            1.\quad & \Omega & \gets& \texttt{randn(n,r)} && \text{(fill $\Omega$ with random values)} \\
            2.\quad & Y & \gets& A\Omega && \text{(matrix multiply)} \\
            3.\quad & Q & \gets& \text{orth}(Y) && \text{(e.g.\ QR decomposition)}
          \end{aligned}
        \end{equation*}
  \end{minipage}%
}
\caption{Randomized range finder with fixed rank $r$.}
\label{Fig:RandRangeFinder_fixed}
\end{figure}

% When an oversampling parameter $p\in\mathbb{N}$, $0\leq p\leq r$, is introduced, this procedure can compute a highly accurate approximation to the best-approximating subspace of rank $r-p$~\cite{halkoFindingStructureRandomness2011}.

Instead of prescribing the rank, one may seek a rank $r$ as low as possible while satisfying a given tolerance $\sigma$,
\begin{align}\label{Eq:tolerance}
  \|A-QQ^TA\|_2 \leq \sigma.
\end{align}
The \emph{Adaptive Randomized Range Finder} (Algorithm 4.2 in~\cite{halkoFindingStructureRandomness2011}) constructs the basis vectors successively until a probabilistic stopping criterion indicates that the desired approximation accuracy has been reached. The original formulation uses Gram--Schmidt orthogonalization of successively sampled vectors.

To efficiently use BLAS level-3 routines and achieve the high arithmetic intensity needed to saturate GPUs, a blocked method is required. One variant achieving this is presented in \cite{martinssonRandomizedBlockedAlgorithm2016} and given in Algorithm \ref{Alg:BlockAdaptiveRandomizedRangeFinder_detErr} in this paper. 

\begin{algorithm}[t]
  \caption{Gram--Schmidt Blocked Adaptive Randomized Range Finder \cite{martinssonRandomizedBlockedAlgorithm2016}}
  \label{Alg:BlockAdaptiveRandomizedRangeFinder_detErr}
  \begin{algorithmic}[1]
    \Require A matrix $A\in\mathbb{R}^{m\times n}$, a tolerance $\sigma\in\mathbb{R}$, $\sigma>0$, and a block size $b\in\mathbb{N}^+$.
    \While{$\|A\| > \sigma$}
        \State Fill $\Omega\in\mathbb{R}^{n\times b}$ with values from a standard Gaussian distribution.
        \State $Q_\text{new}, \_ \gets\text{orth}(A\Omega)$ 
        \State $Q_{\text{new}}, \_ \gets \text{orth}((I - QQ^T)Q_{\text{new}})$
        \Comment{reorth.}\label{Alg:BlockAdaptiveRandomizedRangeFinder_detErr:reorth}
        \State $Q \gets \begin{bmatrix}Q & Q_\text{new}\end{bmatrix}$
        \State $B_{\text{new}} \gets Q_\text{new}^T A$\label{Alg:BlockAdaptiveRandomizedRangeFinder_detErr:QTA}
        \State $B \gets \begin{bmatrix}B \\ B_\text{new}\end{bmatrix}$
        \State $A \gets A - Q_\text{new}B_\text{new}$\label{Alg:BlockAdaptiveRandomizedRangeFinder_detErr:residual}
    \EndWhile
    \Ensure $Q$ fulfilling \eqref{Eq:tolerance}.
  \end{algorithmic}
\end{algorithm}

The array containing the matrix $A$ is updated in place to represent the residual $A-QB$, with the block columns of $Q$ and the block rows of $B=Q^TA$ being continuously built up by the iteration. A norm of this residual matrix is used as a deterministic stopping criterion. Typically, the Frobenius norm is chosen for computational efficiency.

Another stopping criterion is devised in~\cite{wenjianyuyugu+andyaohangliEfficientRandomizedAlgorithms}. Here, the residual matrix does not need to be stored and the residual is instead computed as $\|A-QB\|_F^2 = \|A\|_F^2 - \sum_{i}\|B_i\|_F^2$, where the $B_i$ refer to the generated block rows of $B$. This way for sparse matrices fill-in can be avoided. The authors extend the algorithm to reduce the number of passes over $A$.

The surveyed blocked algorithms orthogonalize sampled panels through some form of Gram--Schmidt and rely on reorthogonalization to keep the accumulated basis sufficiently orthogonal, particularly in ill-conditioned regimes.

Neither classical Gram--Schmidt nor its block variants guarantee orthogonality at the level of unit roundoff~\cite{Carson_2022}. Reorthogonalization is therefore typically required, which complicates implementations that overlap CPU and GPU work. Recent reorthogonalized variants with fewer synchronization points give better stability guarantees~\cite{Carson25}, but, to our knowledge, have not yet been applied to randomized range finders.

Householder QR is a standard method for orthogonalization and is implemented in numerical linear algebra libraries such as LAPACK. It avoids the need for reorthogonalization and provides strong bounds on the error in the computed orthogonal factor~\cite{higham2002}. Using the compact WY representation~\cite{schreiberStorageEfficientWY1987}, most of the blocked factorization can be expressed in terms of level-3 BLAS operations.

Blocked Householder QR~\cite{elmrothApplyingRecursionSerial2000} is moreover a mature computational primitive. GPU-accelerated implementations are available in which CPU panel factorization is overlapped with GPU updates and communication~\cite{magma}. These properties make the blocked Householder QR a natural starting point for developing a high-performance blocked adaptive randomized range finder with good numerical stability.

\section{A Recursive QB factorization based on Householder QR}\label{Sec:Prelim}

The goal of this section is to derive an algorithm that partially performs a classic blocked QR factorization 
\begin{gather*}
    A\Omega = QR,
\end{gather*}
using Householder transformations, where $\Omega\in\mathbb{R}^{n\times r}$, $r\leq n$, is a Gaussian random matrix and $A\Omega$ is not formed explicitly as a whole. 
As in the LAPACK routine \texttt{geqrf}, $Q$ is represented by Householder vectors $v_i$ and scalars $\tau_i$.

We are not interested in $R$, but in the QB decomposition approximating $A$, determined by $Q$,
\begin{gather}\label{Eq:QB}
    A \approx QB,\quad Q\in\mathbb{R}^{m\times r},\ B=Q^TA\in\mathbb{R}^{r\times n}.
\end{gather}

However, $r$ is not set a priori, as we aim to develop an adaptive randomized range finder. We compute the Householder reflector representations successively and check the subspace accuracy \eqref{Eq:tolerance} at each step, and stop as soon as it reaches the desired tolerance.

The compact WY representation~\cite{schreiberStorageEfficientWY1987} is used to represent block reflectors
\begin{gather}\label{Eq:StorEffQR}
    Q_i = I_m - V_i T_i V_i^T,
\end{gather}
where $V_i\in\mathbb{R}^{m\times b}$ contains Householder vectors and $T_i\in\mathbb{R}^{b\times b}$ is an upper triangular matrix.  $Q$ in \eqref{Eq:QB} is given by the first $r$ columns of the matrix product of block reflectors $\prod_{i} Q_i$.
Applying $Q$ resolves into a series of matrix--matrix products.

% In our implementation, these are also put along with the scalar factors $\tau_i$. They can be used in an efficient the application of the block reflectors. 

The key question to design a blocked adaptive range finder algorithm based on the Householder QR decomposition is the following: Given a block reflector from the QR decomposition of a sampled panel $A\Omega_0$, how can it be used to update $A$, such that the next block reflector can be computed from a QR decomposition of the next sampled panel $A\Omega_1$? The following observation clarifies this issue. 

\begin{lemma}\label{Lem:BlockQROmega}
Let $A\in\mathbb{R}^{m\times n}$ and $\Omega = \begin{bmatrix}
        \Omega_1 & \Omega_2
    \end{bmatrix}\in\mathbb{R}^{n\times r}$, where $\Omega_1$ contains the first $b$ columns of $\Omega$. Given the QR factorization
\begin{gather*}
Q_1\begin{bmatrix}
    R_1\\0
\end{bmatrix}:= A\Omega_1,\quad Q_1\in\mathbb{R}^{m\times m},\ R_1\in\mathbb{R}^{b\times b},
\end{gather*}
the updated $A$
\begin{gather*}
    \tilde{A} = \begin{bmatrix}\tilde{A_1}\\ \tilde{A_2}\end{bmatrix}
    := Q_1^TA,\quad \tilde{A}_1 \in\mathbb{R}^{b\times n},\ \tilde{A}_2\in\mathbb{R}^{(m-b)\times n},
\end{gather*}
and the $QR$ decomposition
\begin{gather*}
    Q_2R_2:= \tilde{A}_2\Omega_2, \quad Q_2\in\mathbb{R}^{(m-b)\times (m-b)},\ R_2\in\mathbb{R}^{(m-b)\times (r-b)},
\end{gather*}
then
\begin{align}\label{Eq:Q_product}
    Q:=Q_1\begin{bmatrix}I_b&\\&Q_2\end{bmatrix}
\end{align}
is the orthogonal factor of a QR factorization of $A\Omega$, i.e. $Q^TA\Omega=R$ is upper triangular. 
\end{lemma}
\begin{proof} Simply substituting definitions yields
    \begin{align*}
        Q^TA\Omega 
        &= \begin{bmatrix}R_1 & \tilde{A}_1 \Omega_2\\0&R_2\end{bmatrix}.
    \end{align*}
\end{proof}
In factored form representation, the product of orthogonal reflectors \eqref{Eq:Q_product} is not formed explicitly, but the Householder vectors and their corresponding scaling factors and/or triangular block factors $T_i$ are simply concatenated in a list.  Lemma \ref{Lem:BlockQROmega} provides a recipe to recursively perform the classic randomized range finder introduced in \autoref{Sec:Background}, given in \autoref{Fig:RecursiveQB}, exposing level-3 BLAS operations. 

\begin{figure}[H]
\centering
\fbox{%
  \begin{minipage}{0.95\linewidth}
  \begin{enumerate}
    \item Sample a matrix panel $Y=A\Omega_1$.
    \item Factorize the panel $Q_1R_1 = Y$.
    \item Update $A\gets Q_1^TA$, retrieve $B_1$ as upper block-row.
    \item Repeat for the lower part of $A$ with a new random matrix.
  \end{enumerate}
  \end{minipage}%
}
\caption{Recursive QB factorization based on implicit Householder QR factorization of $A\Omega$.}
\label{Fig:RecursiveQB}
\end{figure}

Similar to the standard blocked QR factorization, the height of the processed panel decreases with each iteration, as does the height of the matrix that needs to be updated. In contrast to the standard QR factorization, the width of the matrix that needs to be updated does not decrease.  It is not possible to store all algorithm matrix outputs in place of $A$, because $B$ is not upper triangular. During the process, building up $B$ in place of $A$ is the natural approach. The Householder vectors need to be stored in an extra array.

The upper part of the updated matrix, $\tilde{A}_1$ in Lemma \ref{Lem:BlockQROmega}, is the upper part of $B$ in the resulting $A=QB$ factorization. The lower part $\tilde{A}_2$ can be used to compute the residual, in order to check if the desired tolerance has been achieved. Split up the columns of $Q_1$ from Lemma \ref{Lem:BlockQROmega} as $Q_1=\begin{bmatrix}Q_{11} & Q_{12}\end{bmatrix}$, then it holds
\begin{gather*}
    \|A  - Q_{11}\tilde{A}_1\|_2 = \|Q_{12}\tilde{A}_2\|_2 = \|\tilde{A}_2\|_2.
\end{gather*}
% The blocked adaptive randomized range finder developed in the next section is based on the following intuition: 
The Frobenius norm can be used as a cheaper, parallelizable alternative to the spectral norm.

Alternatively, the computed row panels of $B$ can be used to provide an even cheaper error criterion as proposed in \cite{wenjianyuyugu+andyaohangliEfficientRandomizedAlgorithms}, as, with the notation from Lemma \ref{Lem:BlockQROmega},
\begin{gather*}
    \|A  - Q_{11}\tilde{A}_1\|_F^2 = \|A\|_F^2 - \|\tilde{A}_1\|_F^2.
\end{gather*}
To utilize this algorithmically, a scalar $E=\|A\|_F^2$ is initialized and the squared Frobenius norms of the computed panels of $B$ are subtracted until $\sigma^2$ is reached. This ensures that $\|A-QB\|_F\leq \sigma$ and hence \eqref{Eq:tolerance}. Numerical limitations of this approach are discussed in Section \ref{Sec:Results}.

\section{The Householder Blocked Adaptive Randomized
Range Finder}\label{Sec:Algorithm}
We divide $A$ into blocks and store Householder vectors in $V$, which is lower triangular with ones on the diagonal. Block rows of $B$ are computed successively and stored in the memory location of $A$. Each storage-efficient factorization \cite{schreiberStorageEfficientWY1987} of a sampled block column yields an upper triangular matrix block (see \eqref{Eq:StorEffQR}), all of which are stored in $T$.

As a notation for referring to blocks, block rows and block columns we use
\newcommand{\rotvert}{\rotatebox[origin=c]{90}{$\vert$}}
\begin{gather*}
A = \begin{bmatrix} 
  A_{0,0} & \cdots & A_{0,k}\\
  \vdots & \ddots & \vdots\\
  A_{j,0} & \cdots & A_{j,k}
\end{bmatrix},\ 
V = \begin{bmatrix} 
  V_{0,0} & \cdots & V_{0,k}\\
  \vdots & \ddots & \vdots\\
  V_{j,0} & \cdots & V_{j,k}
\end{bmatrix},\\  
B=
\begin{bmatrix}
\rotvert&B_0 &\rotvert \\
  &\vdots&\\
  \rotvert & B_{j} & \rotvert
\end{bmatrix},\ 
T = \begin{bmatrix}
T_0 & \cdots& T_k
\end{bmatrix}.
\end{gather*}

We use colon notation to  refer to a submatrix of a matrix $M$ as $M_{i:l,p:q}$. 
The orthogonal subspace basis in \eqref{Eq:QB} is represented by the first $r$ columns of $Q_\text{full} = \prod_{i=0}^k(I_m - V_{0:j,i} T_i V_{0:j,i}^T)$. For simpler notation we assume the matrix dimensions to be divisible by the block size $b$, and assume all blocks to have dimensions $b\times b$.

\begin{algorithm}[t]
  \caption{Householder Blocked Adaptive Randomized Range Finder \label{Alg:QRAdaptiveRandomizedRangeFinder}}
  \begin{algorithmic}[1]
    \Require A matrix $A \in \mathbb{R}^{m \times n}$, a tolerance $\sigma$, and a block size $b$.
    \State $E \gets \|A\|_F^2$\label{Alg:QRAdaptiveRandomizedRangeFinder:Einit}
    \State $B \gets A$
    \State Initialize $V$ with all zeros.
    \State $i \gets 0$
    \While{$E > \sigma^2$}
      \State Fill $\Omega \in \mathbb{R}^{n \times b}$ with values from a standard Gaussian distribution.
      \State $(V_{i:j, i}, T_i) \gets \text{qr}(B_{i:j} \Omega)$
      \Comment{\texttt{geqrt}}\label{Step:geqrt}
      \State $B_{i:j} \gets (I - V_{i:j,i} T_i^T V_{i:j,i}^T) B_{i:j}$ \label{Alg:Bupdate}
      \State $E \gets E - \|B_{i}\|_F^2$ \label{Alg:QRAdaptiveRandomizedRangeFinder:E}
      \State $i \gets i + 1$
    \EndWhile
    \State $V \gets V_{:, 0:i-1}$
    \State $B \gets B_{0:i-1, :}$
    \State $r \gets i\cdot b$
    \Ensure Rank $r$, Householder vectors $V \in \mathbb{R}^{m \times r}$, $B \in \mathbb{R}^{r \times n}$, $T_0, \dots, T_{i-1} \in \mathbb{R}^{b \times b}$ such that $\|A - QB\|_{F} \leq \sigma$, where $Q$ contains the first $r$ columns of $Q_\text{full} = \prod_{l=0}^{i-1} (I_m- V_{0:j,l} T_l V_{0:j,l}^T)$.
  \end{algorithmic}
\end{algorithm}

Algorithm \ref{Alg:QRAdaptiveRandomizedRangeFinder} successively creates the block columns of $V$ and block rows of $B$. Step \ref{Step:geqrt} performs a QR decomposition of a sampled panel, resulting in a compact WY representation \eqref{Eq:StorEffQR}, as given by LAPACK routine \texttt{geqrt}.  

In Algorithm \ref{Alg:QRAdaptiveRandomizedRangeFinder}  the error criterion is based on the panel norms $\|B_i\|_F$.  The alternative variant, using the lower part of $B$, computes $E\gets \|B_{i+1:j}\|_F$ instead of step \ref{Alg:QRAdaptiveRandomizedRangeFinder:E}. The initialization of $E$ in step \ref{Alg:QRAdaptiveRandomizedRangeFinder:Einit} is not needed then, and the loop condition changes to $E>\sigma$ instead of $E>\sigma^2$.

After $\min(j+1,k+1)$ steps, $E$ is zero in exact arithmetic and the loop is guaranteed to end. In case of $j>k$, the last $j-k$ block rows of $B$ are zero. If the stopping criterion is met before all $j+1$ block rows of B or all $k+1$ block columns of $V$ have been computed, $B$, $V$ and $T$ are truncated resulting in a low rank decomposition.

\section{Hybrid CPU--GPU Implementation}\label{Sec:Implementation}
Our goal is to develop a hybrid CPU–GPU implementation of Algorithm~\ref{Alg:QRAdaptiveRandomizedRangeFinder} that keeps the GPU saturated with high-arithmetic-intensity BLAS-3 operations, and minimizes host-device data transfer. 
We first develop a variant where CPU and GPU work are not overlapped. 
All input matrices are assumed to be available on the GPU, as this reflects the circumstances in our motivating example from deep learning, where weights and gradients reside on the GPU during the training of neural networks. 

In Algorithm \ref{Alg:QRAdaptiveRandomizedRangeFinder}, most operations are suitable to run on the GPU. The random matrix is generated with cuRAND and the panel sampling $B_{i:j}\Omega$ uses cuBLAS. The update of $B$  consists of matrix--matrix multiplications, which are realized using the MAGMA library~\cite{magma_solvers} with cuBLAS as a back end. 

The panel is factorized on the CPU using a LAPACK implementation. To this end, the sampled panel is transferred to the CPU, and the results of the factorization, i.e. the Householder vectors and the $T_i$ matrix, are transferred back to the GPU and utilized for the update of $B$.  
The described implementation is visualized as a task graph in \autoref{Fig:Sequential}. 

\begin{figure}[t]
\centering
\tikzset{
    box/.style={
        rectangle,
        rounded corners=6pt,
        draw=black!60,
        minimum width=2cm,
        minimum height=0.9cm,
        text width=1.8cm,
        align=center,
        font=\small,
        line width=0.9pt
    },
    samplebox/.style={
        box,
        draw=blue!60!black,
        fill=blue!12
    },
    factorbox/.style={
        box,
        draw=teal!60!black,
        fill=teal!12
    },
    updatebbox/.style={
        box,
        draw=orange!70!black,
        fill=orange!18
    },
    updatepanelbox/.style={
        box,
        draw=red!70!black,
        fill=yellow!25,
        line width=1.2pt
    },
    flow/.style={
        -{Latex[length=2.2mm]},
        draw=black!60,
        line width=0.9pt
    }
}

%-----------------------------------
% Subfigure 1
%-----------------------------------
\begin{subfigure}[t]{0.4\columnwidth}
\centering
\begin{tikzpicture}[scale=0.75, transform shape]

\node[samplebox] (sample) {Sample Panel 0};
\node[factorbox, below=0.5cm of sample]
    (factor0) {Factor Panel 0};
\node[updatebbox, below=0.5cm of factor0]
    (update) {Update $B$};
\node[samplebox, below=0.5cm of update] (sample1) {Sample Panel 1};
\node[factorbox, below=0.5cm of sample1]
    (factor1) {Factor Panel 1};

\node[below=0.45cm of factor1, font=\Large]
    (dots1) {$\vdots$};

\draw[flow] (sample) -- (factor0);
\draw[flow] (factor0) -- (update);
\draw[flow] (update) -- (sample1);
\draw[flow] (sample1) -- (factor1);
\draw[flow] (factor1) -- (dots1);

\end{tikzpicture}
\caption{Sequential implementation.}
\label{Fig:Sequential}
\end{subfigure}
\hspace{0.06\columnwidth}
%-----------------------------------
% Subfigure 2
%-----------------------------------
\begin{subfigure}[t]{0.45\columnwidth}
\centering
\begin{tikzpicture}[scale=0.75, transform shape]

% Top row
\node[samplebox, draw=none,fill=none] (dummy0) {};

\node[samplebox, right=0.45cm of dummy0] (sample0) {Sample Panel 0};

% \node[
%     samplebox,
%     draw=none,
%     fill=none,
%     right=0.45cm of sample0
% ] (dummy0) {};

% Second row
\node[samplebox, below=0.5cm of dummy0]
    (sample1) {Sample Panel 1};

\node[factorbox, below=0.5cm of sample0]
    (factor0) {Factor Panel 0};

% Third row
\node[updatebbox, below=0.5cm of sample1]
    (updateB0) {Update $B$};

\node[updatepanelbox, below=0.5cm of factor0]
    (updatepanel1) {Update Panel 1};

% Fourth row
\node[samplebox, below=0.5cm of updateB0]
    (sample2) {Sample Panel 2};

\node[factorbox, below=0.5cm of updatepanel1]
    (factor1) {Factor Panel 1};

% Fifth row
\node[updatebbox, below=0.5cm of sample2]
    (updateB1) {Update $B$};

\node[updatepanelbox, below=0.5cm of factor1]
    (updatepanel2) {Update Panel 2};

% Continuation dots
\node[below=0.45cm of updateB1, font=\Large]
    (dots1) {$\vdots$};

\node[below=0.45cm of updatepanel2, font=\Large]
    (dots2) {$\vdots$};

% Queue labels with identical top alignment
\coordinate (queueLabelRight)
    at ([yshift=1cm]sample0.north);

\coordinate (queueLabelLeft)
    at (queueLabelRight -| sample1.center);

\node[
    anchor=north,
    font=\small\bfseries,
    align=center
] at (queueLabelLeft)
    {Queue 0};

\node[
    anchor=north,
    font=\small\bfseries,
    align=center
] at (queueLabelRight)
    {Queue 1};

% Flow arrows
\draw[flow] (sample0) -- (factor0);

\draw[flow] (sample1) -- (updateB0);
\draw[flow] (sample1) -- (updatepanel1);

\draw[flow] (factor0) -- (updateB0);
\draw[flow] (factor0) -- (updatepanel1);

\draw[flow] (updateB0) -- (sample2);
\draw[flow] (updatepanel1) -- (factor1);

\draw[flow] (sample2) -- (updateB1);
\draw[flow] (sample2) -- (updatepanel2);

\draw[flow] (factor1) -- (updateB1);
\draw[flow] (factor1) -- (updatepanel2);

\draw[flow] (updateB1) -- (dots1);
\draw[flow] (updatepanel2) -- (dots2);

\end{tikzpicture}
\caption{Exposed parallelism by introduced panel update.}
\label{Fig:Parallel}
\end{subfigure}
\caption{Task-dependency graphs for two implementations of Algorithm \ref{Alg:QRAdaptiveRandomizedRangeFinder}. In implementation (a), the next panel cannot be sampled until the full update of $B$ is complete. In the overlapped implementation (b), the sampled panel is updated independently, introducing extra operations while shortening the critical path.}
\label{Fig:DAG}
\end{figure}
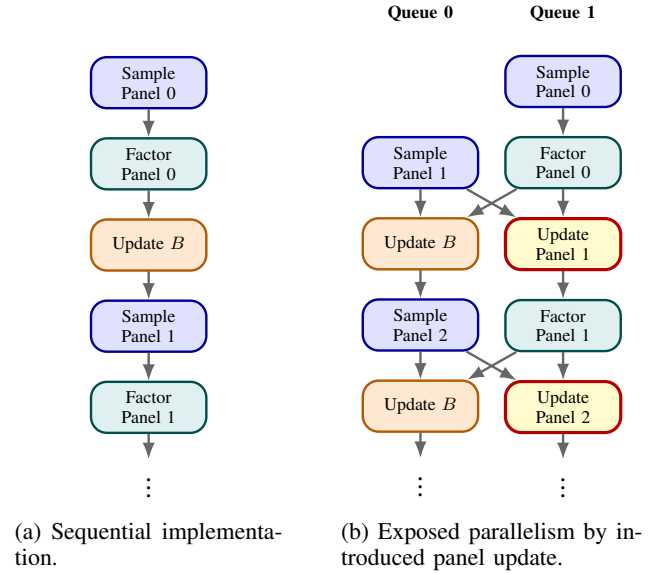

Following established approaches for dense linear algebra solvers in hybrid CPU–GPU environments, we aim to overlap the panel factorization on the CPU with the update step on the GPU. In classical Householder QR, this overlap arises naturally by partitioning the trailing update into two parts: the update of the next panel and the update of the remaining matrix. Since only the next panel must be updated before the subsequent panel factorization can begin, the CPU can proceed with that factorization while the GPU completes the larger update of the remainder.

For the proposed blocked adaptive randomized range finder (Algorithm \ref{Alg:QRAdaptiveRandomizedRangeFinder}) this is not directly possible because the entire updated matrix $B$ is required to compute the next sampled panel. However, at the cost of additional operations, the critical path can be shortened by introducing redundant computations. Instead of waiting for the complete update of $B$ (step \ref{Alg:Bupdate}), we can compute a sampled panel $B\Omega$ and update it separately. This redundancy allows the panel update and factorization, on the one hand, to proceed independently of the matrix update and panel sampling, on the other. \autoref{Fig:Parallel} illustrates the resulting dependency graph of this overlapping variant. Here, ``Update Panel'' denotes the additional operations required compared with the variant shown in \autoref{Fig:Sequential}. Among the operations, represented by nodes in \autoref{Fig:DAG}, only the panel factorizations are performed on the CPU. 

As the algorithm shows structural commonalities with the standard Householder QR decomposition, the MAGMA routine \texttt{sgeqrf3\_gpu} has served as a basis for implementation.  In order to achieve CPU--GPU overlap, a non-blocking routine to compute a squared Frobenius norm had to be devised to compute the panel norm $\|B_i\|_F^2$ in step \ref{Alg:QRAdaptiveRandomizedRangeFinder:E}.

\section{Numerical Experiments}\label{Sec:Results}
First, we perform experiments to assess approximation quality and numerical stability of our proposed Householder (HH) approach (Algorithm \ref{Alg:QRAdaptiveRandomizedRangeFinder}) and compare it to range finders based on Block Classical Gram--Schmidt (BCGS) (Algorithm \ref{Alg:BlockAdaptiveRandomizedRangeFinder_detErr}), where a QR decomposition is used as inner block orthogonalization. We compare variants with and without reorthogonalization (step \ref{Alg:BlockAdaptiveRandomizedRangeFinder_detErr:reorth} in Algorithm \ref{Alg:BlockAdaptiveRandomizedRangeFinder_detErr}) to assess its impact. For these experiments, the algorithms were implemented in Python.

A $1024\times1024$ test matrix is constructed with prescribed singular values logarithmically spaced from $10^0$ to $10^{-14}$ and random orthogonal left and right singular vectors. All methods use a block size of $b=64$ and perform calculations in single precision. After each block, we measure the spectral-norm projection residual $\|A-QQ^TA\|_2$, where $Q$ contains all computed basis vectors in explicit form up to that point. For the diagnostic purposes of this experiment, a thin $Q$ is constructed from Householder vectors using the \texttt{sorgqr} LAPACK routine. The stopping quantities tracked by the respective methods are reported as well. For the BCGS-based method, this refers to the Frobenius norm of the residual matrix $A - QB$ that is tracked in place of the original $A$ (step \ref{Alg:BlockAdaptiveRandomizedRangeFinder_detErr:residual} in Algorithm \ref{Alg:BlockAdaptiveRandomizedRangeFinder_detErr}).
For the Householder-based method, this refers to the scalar $E^{\frac{1}{2}}$, which aims to track the norm of the residual $\|A-QB\|_F$ through $E=\|A\|_F^2 - \sum_{i}\|B_i\|_F^2$ (step \ref{Alg:QRAdaptiveRandomizedRangeFinder:E} in Algorithm \ref{Alg:QRAdaptiveRandomizedRangeFinder}). 

In accordance with the floating point analysis in \cite{wenjianyuyugu+andyaohangliEfficientRandomizedAlgorithms}, we observe a breakdown of the panel-based error criterion at $k=384$, going below 0, when the reported value $E^{\frac{1}{2}}$ reaches a level around $10^{-4}$. At this point, $E$ is obtained by subtracting two quantities of order $\|A\|_F^2$ whose difference is close to floating-point machine precision. This cancellation leads to a loss of significant digits. We switch to reporting the residual based on the lower part of $B$ at $E^{\frac{1}{2}}<5\cdot 10^{-3}\|A\|_F$.
The loss of orthogonality $\|Q^TQ-I\|_2$ is also reported. All results are found in \autoref{Fig:orthogonality_loss}.

\begin{figure*}[t]
    \centering
    \includegraphics[width=\textwidth]{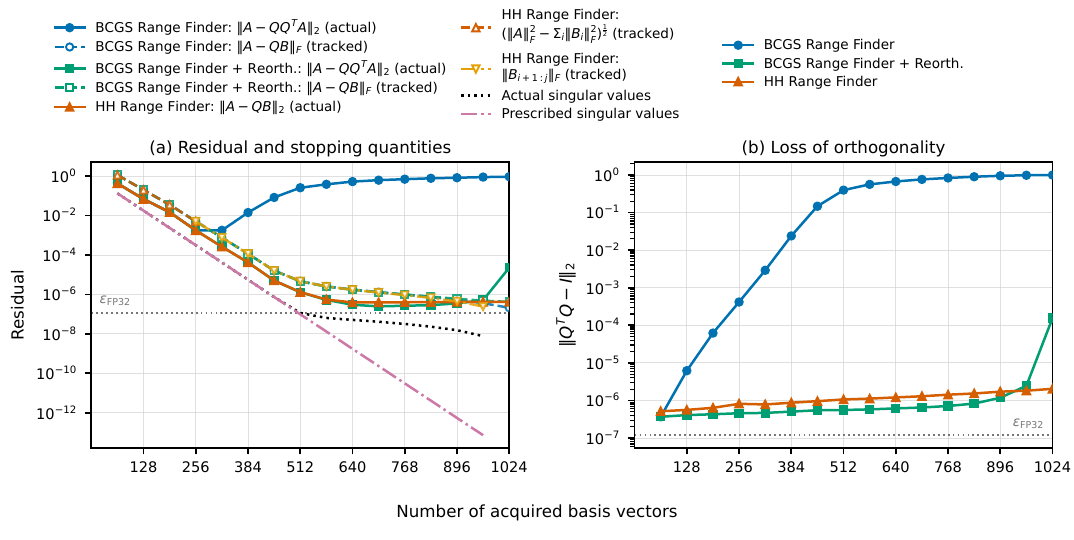}
    \caption{Residual and stopping quantities (a) and loss of orthogonality (b) for BCGS range finder (Algorithm \ref{Alg:BlockAdaptiveRandomizedRangeFinder_detErr}) with and without reorthogonalization, and the proposed Householder-based method (HH range finder, Algorithm \ref{Alg:QRAdaptiveRandomizedRangeFinder}). The actual and prescribed values of $\sigma_{k+1}(A)$ are included as reference bounds, and the horizontal dotted line marks FP32 unit roundoff.}
    \label{Fig:orthogonality_loss}
\end{figure*}

We see in \autoref{Fig:orthogonality_loss} that initially the residual goes down for all methods as more basis vectors are added. Without reorthogonalization, the orthogonality for the BCGS variant starts to deteriorate immediately, and soon this is reflected in the residual $\|A-QQ^TA\|_2$. The tracked stopping quantity $\|A - QB\|_F$ stays low, as $B$ absorbs $Q$'s defects. Reorthogonalization remedies these issues.

The residuals of both methods approach the noise floor introduced by single precision. The prescribed singular values go below machine precision, making the resulting matrix numerically low-rank. The computable, actual singular values of rounded $A$ represent the best approximation possible by the given number of basis vectors. Even with reorthogonalization, BCGS accuracy deteriorates near the end, as no meaningful information is incorporated anymore. The tracked stopping criterion fails to capture this. An adaptive range finder should stop before this point. Our Householder variant yields a low residual, and preserves orthogonality, even under these extreme circumstances, with no need for reorthogonalization. 

We next evaluate the runtime performance of the proposed algorithm and implementation in the experiments reported in Figs.~\ref{Fig:full_barrf}, \ref{Fig:randQB}, and \ref{Fig:fixed}. The algorithms and driving routines are implemented in C++ and utilize optimized libraries. In these experiments we generate random square matrices and run the adaptive range finder to completion, i.e., until a full basis has been computed. We use the panel-based error criterion. The previous experiment showed that, due to its numerical limitations, this criterion is not suitable for production runs, when an approximation tolerance below around $5\cdot 10^{-3}\|A\|_F$ in single precision is desired. We refer to the error analysis in \cite{wenjianyuyugu+andyaohangliEfficientRandomizedAlgorithms} and leave an implementation of the error criterion based on the lower part of $B$, where needed, as future work. 

Runtimes were measured on a GH200 node of the JURECA-DC Evaluation Platform\footnote{\url{https://apps.fz-juelich.de/jsc/hps/jureca/evaluation-platform-overview.html\#grace-hopper-nodes} -- accessed 2026/08/08.}, containing NVIDIA’s Grace Hopper Superchip. The Grace CPU has 72 Arm cores and the GPU is a tightly coupled Hopper GPU. The node contains 480 GiB LPDDR5X and 96 GiB HBM3 memory. The NVPL library was used for all CPU-based LAPACK and BLAS calls. All experiments were performed in single precision. For Figs.~\ref{Fig:full_barrf} and \ref{Fig:randQB}, the median runtime of 5 runs after a warmup round is reported. For \autoref{Fig:fixed}, the mean of 3 runs after warmup is reported. 

\autoref{Fig:full_barrf} shows a heat map of runtimes achieved by variants of our proposed Algorithm \ref{Alg:QRAdaptiveRandomizedRangeFinder}, with varying matrix size $m=n$ and block size $b$.  The first heat map refers to a CPU-only variant, where panel factorization and matrix update are not overlapped. The center heat map shows runtime results for a GPU accelerated variant of Algorithm \ref{Alg:QRAdaptiveRandomizedRangeFinder}, where the random matrix generation, the GEMM to sample the panel, the matrix update (step \ref{Alg:Bupdate}), and the computation of $\|B_i\|_F^2$ are realized on the GPU, using cuRAND and MAGMA, with a cuBLAS back end. No extra panel update has been introduced so that CPU and GPU work are not overlapped (see \autoref{Fig:Sequential}). The third heat map shows the performance of the two-queue variant (see \autoref{Fig:Parallel}), which introduces additional operations to enable overlap between CPU and GPU work. Both GPU-accelerated variants assume that the input matrix resides in GPU memory and leave the output there. This setup reflects our motivating machine-learning use case, in which the relevant matrices remain on the accelerator throughout the computation.

\begin{figure*}[t]
    \centering
    \includegraphics[width=\textwidth]{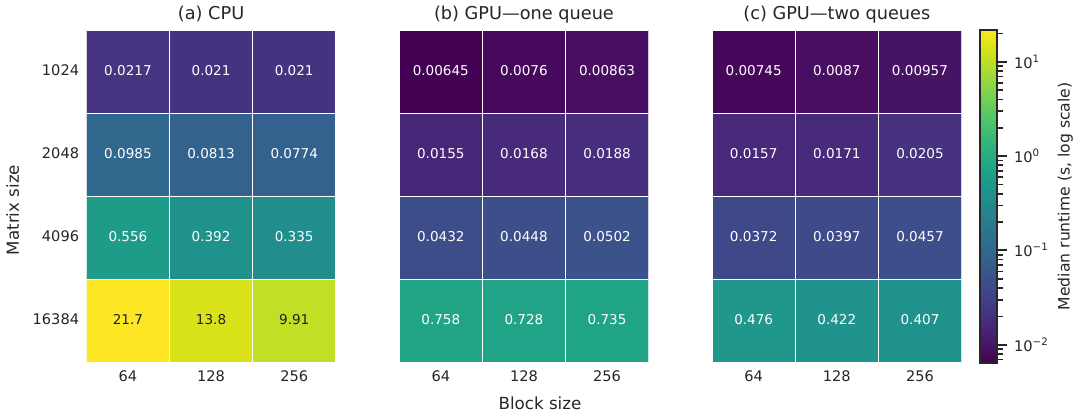}
    \caption{Measured runtimes for full runs of the Householder blocked adaptive range finder (Algorithm \ref{Alg:QRAdaptiveRandomizedRangeFinder}) on a GH200.}
    \label{Fig:full_barrf}
\end{figure*}

\autoref{Fig:full_barrf} shows that the GPU implementations substantially outperform the CPU implementation, with the advantage increasing with matrix size. For small matrices, the one- and two-queue variants perform similarly, and the additional overlap provides little benefit. For larger matrices, however, the two-queue implementation becomes clearly advantageous, reducing the runtime for the largest case from \qty{9.91}{\second} on the CPU and \qty{0.728}{\second} with one GPU queue to \qty{0.407}{\second}. This indicates that the additional parallelism exposed by overlapping panel processing with the matrix update becomes increasingly effective as the problem size grows.

A GPU-accelerated implementation of Algorithm~\ref{Alg:BlockAdaptiveRandomizedRangeFinder_detErr} is available as part of RSVDPACK~\cite{martinssonRandomizedBlockedAlgorithm2016,voronin2016rsvdpackimplementationrandomizedalgorithms}. We forked and adapted the publicly available code\footnote{\url{https://github.com/sergeyvoronin/LowRankMatrixDecompositionCodes} -- accessed 2026/08/08} to run on the GH200 system and to enable a meaningful comparison with the experiments in \autoref{Fig:full_barrf}. In addition to modifications to the test routines, we replaced the explicit MKL dependency with NVPL and changed the GPU-accelerated implementation from double to single precision.

\begin{figure}[t]
    \centering
    \includegraphics[width=\linewidth]{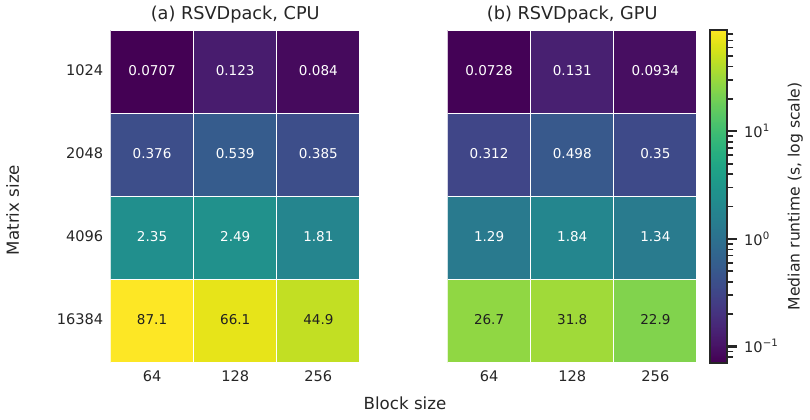}
    \caption{Measured runtimes for computing the QB factorization in RSVDPACK~\cite{voronin2016rsvdpackimplementationrandomizedalgorithms} on a GH200, including host-device and device-host transfers.}
    \label{Fig:randQB}
\end{figure}

\autoref{Fig:randQB} shows results for the RSVDPACK implementation of Algorithm~\ref{Alg:BlockAdaptiveRandomizedRangeFinder_detErr} using the same matrix sizes and block sizes as in \autoref{Fig:full_barrf}. Its GPU-accelerated version is at most about $3\times$ as fast as the CPU version, and both are substantially slower than the corresponding variants of our proposed Householder-based adaptive range finder, Algorithm~\ref{Alg:QRAdaptiveRandomizedRangeFinder}. For $m=n=16\,384$, our GPU-accelerated implementation achieves a speedup of up to $56\times$ over the RSVDPACK implementation when comparing the best configuration of each. 

The BCGS-based routines in RSVDPACK are not optimized and leave substantial room for improvement. Matrices reside in CPU memory, and for each matrix multiplication data transfers to and from the GPU are performed, introducing a large overhead. The presented results do not establish an inherent performance advantage of the Householder approach; however, no other GPU implementation of the BCGS method is currently available to our knowledge. Without reorthogonalization, the BCGS-based method could adopt a look-ahead scheme similar to the one used here to overlap CPU and GPU work. With reorthogonalization, which is required to achieve stability comparable to Householder QR (see \autoref{Fig:orthogonality_loss}), additional dependencies make such overlap more difficult.

The experiment reported in \autoref{Fig:fixed} compares the performance of our proposed adaptive range finder, Algorithm~\ref{Alg:QRAdaptiveRandomizedRangeFinder}, with a simple fixed-rank randomized range finder as outlined in \autoref{Fig:RandRangeFinder_fixed}. We implemented three variants of the fixed-rank method, each consisting of random matrix generation, a matrix multiplication to form the sampled range, and a QR factorization. To obtain the same $QB$ output as our proposed method, we compute $B=Q^TA$ by applying the resulting Householder reflectors to $A$. The CPU implementation uses LAPACK and BLAS routines provided by NVPL, while the two GPU implementations use either MAGMA or cuBLAS together with cuSOLVER. For fixed matrix dimensions $m=n=16\,384$ and block size $b=128$, we vary the target rank $r$, i.e., the number of basis vectors to be computed. For a direct runtime comparison, Algorithm~\ref{Alg:QRAdaptiveRandomizedRangeFinder} is run in a fixed-rank mode: instead of terminating when the prescribed approximation tolerance is reached, it stops after computing $r$ basis vectors.

\begin{figure}[t]
    \centering
    \includegraphics[width=\linewidth]{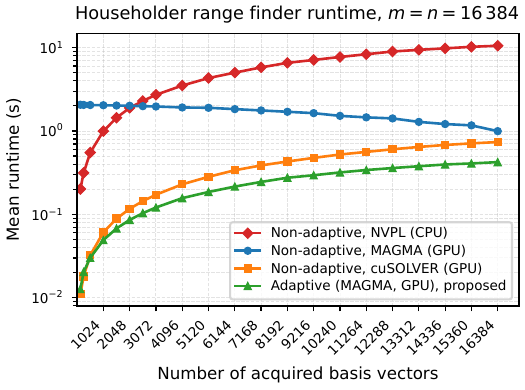}
    \caption{Runtimes of non-adaptive randomized range finder variants, compared to the adaptive range finder with a prescribed number of basis vectors, a block size of $b=128$ on a GH200.}
    \label{Fig:fixed}
\end{figure}

The main advantage of the adaptive range finder is that it determines the required rank from a prescribed approximation tolerance. A fixed-rank method may compute more basis vectors than necessary to achieve the same accuracy, providing a source of potential runtime savings for the adaptive approach. \autoref{Fig:fixed} shows that, even when this advantage is removed by computing the same number of basis vectors, our proposed method outperforms the naive fixed-rank variants. We attribute this performance advantage to the interleaving of operations in the proposed algorithm. This structure enables more effective utilization of the GPU, making our implementation approach also attractive in a fixed-rank setting.

For small target ranks ($r=128$ and $256$), the fixed-rank range finder using cuSOLVER achieves the best performance. In this regime, Algorithm~\ref{Alg:QRAdaptiveRandomizedRangeFinder} is dominated by the CPU panel factorization and by the comparatively inefficient construction of the $T$ matrix. The MAGMA-based fixed-rank implementation exhibits a different and surprising trend: its runtime decreases as the target rank increases. This behavior results from a particularity of MAGMA's QR routines. During the \texttt{sgeqrf} factorization, MAGMA does not construct a $T$ matrix for the final panel, since the corresponding Householder block reflector is not needed for a trailing-matrix update. Consequently, when this reflector is subsequently applied by \texttt{sormqr}, the $T$ matrix required for a GPU-based \texttt{magma\_{s}larfb} operation is unavailable, and the application falls back to a CPU-based LAPACK routine. For a large matrix and a small target rank, this CPU operation dominates the runtime and leaves the GPU underutilized. As the target rank increases, the relative cost of this final CPU-based application decreases, explaining the observed reduction in runtime.

\section{Towards Efficient Low-Rank Deep Learning}\label{Sec:Galore}

Methods from randomized numerical linear algebra show promise as viable tools in the context of resource-efficient low-rank deep learning. 

%Maybe delete this paragraph
In the training of a deep neural network, each layer is represented by matrices, including weights, gradients, and optimizer states. The weights are updated using gradients computed by backpropagation, typically along with optimizer states, e.g., in the popular Adam optimizer. These states encode moving averages of the gradient's first and second moments, incorporating past iteration data to guide updates more effectively.
GaLore reduces optimizer-state memory by representing these states in a low-rank dominant subspace of the gradient matrix~\cite{zhaoGaLoreMemoryEfficientLLM2024}. While GaLore treats the rank $r$ as a hyperparameter, typically chosen heuristically, our adaptive range finder instead allows the approximation tolerance $\sigma$ in \eqref{Eq:tolerance} to serve as the hyperparameter, determining a suitable rank automatically.
With an adaptive method, the dimensionality of subspaces across consecutive training steps can vary. This highlights the problem that adding optimizer states, represented in different subspaces, is not very meaningful and can lead to deteriorated performance, even when the rank is fixed. A linear transformation can be applied to ensure subspace consistency. To realize this, a low-rank optimizer state $M_t\in\mathbb{R}^{r_{t}\times n}$ at step $t+1$ is replaced by $M_{t+1}=Q_{t+1}^TQ_{t}M_t$, where $Q_t\in\mathbb{R}^{m\times r_t}$ contains the subspace basis of the previous step and $Q_{t+1}\in\mathbb{R}^{m\times r_{t+1}}$ contains the current one. 

\section{Conclusion and Outlook}
We presented a Householder-based blocked adaptive randomized range finder that avoids explicit reorthogonalization and maps efficiently to hybrid CPU--GPU execution. On the GH200, the two-queue implementation reduced the runtime of the largest tested case from \qty{9.91}{\second} on the CPU to \qty{0.407}{\second}, demonstrating the benefit of overlapping CPU panel factorization with GPU-based matrix updates.
A natural extension to improve the subspace accuracy for matrices with slowly decaying singular values would be to implement a power scheme, as described in \cite{martinssonRandomizedBlockedAlgorithm2016}.
There is room for further algorithmic and implementation-level optimization. This includes the sensible choice of error criteria. If the projected matrix $B$ is not required, updating its completed block rows may be unnecessary. This could motivate variants that leave $A$ unchanged and instead build up $T$ to provide a compact WY representation of the computed basis~\cite{BenKP17}. 
Whether adaptive randomized range finders based on Gram--Schmidt orthogonalization can achieve competitive performance and numerical stability remains an important question for future work. This includes not only performance-optimized variants of algorithms in RSVDPACK~\cite{voronin2016rsvdpackimplementationrandomizedalgorithms}, but also the integration of recent advances in Gram--Schmidt methods that reduce synchronization requirements~\cite{Carson25}. A systematic comparison of arithmetic cost, communication volume, synchronization requirements, numerical stability, and runtime performance would provide valuable insight.
Applying the proposed method in low-rank deep-learning frameworks such as GaLore will require extensions to distributed matrix layouts arising from multidimensional parallelism. 

Overall, the results in this work demonstrate the potential of the proposed Householder-based formulation as a numerically robust and efficient GPU-oriented approach to adaptive low-rank matrix approximation.

\section{Code Availability and  Acknowledgements}

The code and manual to reproduce all findings in this work are available at \url{https://github.com/caropen/Reproduce-BARRF-Householder}. The GPU-accelerated implementations are available at \url{https://github.com/caropen/Householder-Blocked-Adaptive-Randomized-Range-Finder}. 

ChatGPT and Claude were utilized in the creation of this work, including text editing and code generation.

\bibliographystyle{IEEEtran}
\bibliography{references}

\end{document}